\documentclass[a4paper,12pt]{amsart}

\usepackage[x11names,table]{xcolor}
\usepackage[naturalnames=false]{hyperref}
\usepackage{graphicx}
\graphicspath{ {img/} }
\hypersetup{
	colorlinks = true,          
	urlcolor = DeepSkyBlue3, linkcolor = Chartreuse4, citecolor = DarkOrange2
}
\usepackage{fullpage}
\RequirePackage[backend=bibtex,maxbibnames=5]{biblatex}
\renewbibmacro{in:}{, }
\DeclareFieldFormat[misc]{title}{\mkbibquote{#1}}
\DeclareFieldFormat[article]{volume}{\mkbibbold{#1}}
\DeclareFieldFormat{url}{\textsc{url}: \href{#1}{#1}}
\DeclareFieldFormat{doi}{\textsc{doi}: \href{http://dx.doi.org/#1}{#1}}
\DeclareFieldFormat{eprint:arxiv}{arXiv:\href{https://arxiv.org/abs/#1}{#1}}
\usepackage[utf8]{inputenc}
\usepackage[american]{babel}

\usepackage{amsmath}  
\usepackage{amssymb}     
\usepackage{amsthm}  
\usepackage{bbm}    

\usepackage{bm}
\usepackage{accents}
\usepackage{mathrsfs}
\usepackage{mathtools}
\usepackage{fixmath}
\usepackage{paralist}
\usepackage{enumerate}
\usepackage{xspace}

\usepackage{stmaryrd}

\usepackage{booktabs}
\usepackage{array}
\usepackage{cleveref}
\usepackage{graphicx}
\usepackage{pgf,tikz}
\usepackage{array}
\usepackage{hyperref}
\usepackage{enumitem}
\usepackage{csquotes}

\usepackage{pgf,tikz}
\usetikzlibrary{arrows}
\usetikzlibrary{patterns}
\usetikzlibrary[positioning]
\usetikzlibrary[fit]
\usetikzlibrary{shapes.geometric}
\usetikzlibrary{calc}

\usepackage{pgfplots}
\pgfplotsset{compat=1.17}

\newcommand\KK{{\mathbb K}}
\newcommand\QQ{{\mathbb Q}}
\newcommand\RR{{\mathbb R}}
\newcommand\Sph{{\mathbb S}}

\newcommand\ZZ{{\mathbb Z}}

\newcommand\cF{{\mathcal F}}

\newcommand\cL{{\mathcal L}}
\newcommand\cM{{\mathcal M}}

\newcommand\cB{{\mathcal B}}

\newcommand\biggSetOf[2]{\biggl\{#1 \biggm| #2\biggr\}}

\DeclareMathOperator\conv{conv}
\DeclareMathOperator\cork{cork}

\DeclareMathOperator\Sym{Sym}
\DeclareMathOperator\Aut{Aut}

\DeclareMathOperator\rk{rk}
\DeclareMathOperator\Gr{Gr}
\DeclareMathOperator\TGr{TGr}
\DeclareMathOperator\Dr{Dr}
\DeclareMathOperator\FlDr{FlDr}
\DeclareMathOperator\val{val}

\newcommand\puiseuxseries[2]{{#1}\{\hskip-.25em\{#2\}\hskip-.25em\}}

\newtheorem{theorem}{Theorem}
\newtheorem{lemma}[theorem]{Lemma}
\newtheorem{proposition}[theorem]{Proposition}

\theoremstyle{remark}

\newtheorem{example}[theorem]{Example}

\newtheorem{question}[theorem]{Question}

\newcommand\polymake{\texttt{polymake}\xspace}
\newcommand\TOPCOM{\texttt{TOPCOM}\xspace}
\newcommand\OSCAR{\texttt{OSCAR}\xspace}
\newcommand\mptopcom{\texttt{mptopcom}\xspace}
\newcommand\Gfan{\texttt{Gfan}\xspace}

\graphicspath{ {img/} }

\setdefaultitem{$\triangleright$}{}{}{}

\newcommand \acknowledgements{\section*{Acknowledgements}}

\title{Generalized permutahedra and \\ positive flag Dressians}

\author{Michael Joswig \and Georg Loho \and Dante Luber \and Jorge Alberto Olarte} 

\address[Michael Joswig]{
	Technische Universität Berlin,
	Chair of Discrete Mathematics/Geometry \\
	Max-Planck Institute for Mathematics in the Sciences, Leipzig \\
	\texttt{joswig@math.tu-berlin.de}
}

\address[Dante Luber \and Jorge Alberto Olarte]{
	Technische Universität Berlin,
	Chair of Discrete Mathematics/Geometry \\
	\texttt{$\{$luber,olarte$\}$@math.tu-berlin.de}	
}

\address[Georg Loho]{
  University of Twente,
  Department of Applied Mathematics \\
  \texttt{g.loho@utwente.nl}
  }

\thanks{%
	Supported by the Deutsche Forschungsgemeinschaft (DFG, German Research Foundation) \enquote{The Berlin Mathematics Research Center MATH$^+$} (EXC-2046/1, project ID 390685689).
	MJ has further been supported by \enquote{Symbolic Tools in Mathematics and their Application} (TRR 195, project-ID 286237555).
	DL was supported by \enquote{Facets of Complexity} (GRK 2434, project-ID 385256563).
      }

\begin{document}

\begin{abstract}
  We study valuated matroids, their tropical incidence relations, flag matroids and total positivity.
  This leads to a characterization of \emph{permutahedral subdivisions}, namely subdivisions of regular permutahedra into generalized permutahedra. 
  Further, we get a characterization of those subdivisions arising from \emph{positive} valuated flag matroids.
\end{abstract}

\maketitle

\section{Introduction}

A  \emph{generalized permutahedron} in $\RR^n$ is a convex polytope such that all its edges are parallel to $e_i-e_j$ for $i\neq j$. 
Here $e_1,e_2,\dots,e_n$ denote the standard basis vectors of $\RR^n$.
Generalized permutahedra were introduced by Postnikov \cite{Postnikov:2009}, but they were known before under different names, most notably in the context of submodular function optimization~\cite{Fujishige:2005} and discrete convex analysis~\cite{Murota:2003}. 
Relevant examples of generalized permutahedra include the regular permutahedra, Bruhat interval polytopes, hypersimplices and more general matroid polytopes.
In this way the topic is connected to group and representation theory, algebraic and tropical geometry, optimization and beyond.

Here we continue the combinatorial study of polyhedral subdivisions of generalized permutahedra into cells which are again generalized permutahedra \cite{ArdilaDoker:2013,ArdilaCastilloEurpostnikob:2020,BrandtEurZhang:2021}.
We call these \emph{permutahedral subdivisions}.
An important special case are the tropical linear spaces, which amount to regular subdivisions of matroid polytopes with matroidal cells; see \cite[\S4.4]{Tropical+Book} and \cite[\S10.5]{ETC}.
Our first main contribution (\Cref{thm:regular-2-skeleton}) is a description of regular permutahedral subdivisions of the regular permutahedra in terms of conditions on the $2$-skeleton.
This result is an adaptation of the characterization of the uniform tropical linear spaces via the $3$-term Plücker relations.

One motivation for research in this direction comes from the wish to understand flags of tropical linear spaces and flag matroids.
As a new tool, we prove a characterization of pairs of valuated matroids in terms of tropical incidence relations (\Cref{thm:1step}).
Combined with~\cite{BrandtEurZhang:2021}, we gain new insights in the flag Dressian, the space of incident valuated matroids.
Speyer and Williams~\cite{SpeyerWilliams:2021} and, independently, Arkani-Hamed, Lam, and Spradlin \cite{ArkaniHamedLamSpradlin:2021} recently showed that the positive Dressian equals the positive tropical Grassmannian.
Our second main result (\Cref{thm:positivity}) is related, as it shows that the valuated flags which can be realized by a totally positive flag of linear spaces correspond to the permutahedral subdivisions whose cells are Bruhat interval polytopes.

\section{The geometry of the regular permutahedron}
\label{sec:geometry-regular-permutahedron}

We denote the symmetric group of degree $n$ as $\Sym(n)$.
The (regular) permutahedron $\Pi_n\subseteq \RR^n$ is defined as the convex hull of the points $(\sigma(1),\sigma(2),\dots,\sigma(n))$, where $\sigma$ ranges over $\Sym(n)$. 
These points form the $n!$ vertices of $\Pi_n$, and $\Sym(n)$ acts on this set, e.g., by multiplication on the right.
We have $\dim\Pi_n=n-1$.
Throughout, we identify $\Sym(n)$ with the vertices of $\Pi_n$. 

The face structure of $\Pi_n$ is well-known, and we will use the following description.
A \emph{flag} $(F_1,\dots,F_k)$ in $[n]$ with $k$ \emph{constituents} is a strictly increasing sequence $F_{1}\subset \dots \subset F_{k}$ of non-empty subsets of $[n]$; a flag is \emph{full} if it has $n$ constituents.
We say that a flag \emph{extends} a pair $(U,V)$ of subsets $U \subset V \subseteq [n]$ if there are $i,j \in [k]$ with $U = F_i$ and $V = F_j$.
We set $e_A:= \sum_{i\in A} e_i$.
For the flag $\cF=(F_1,\dots,F_k)$ we can pick real numbers $\lambda_i$ with $\lambda_1 \gg \dots \gg \lambda_k> 0$ to obtain the vector 
\[
\lambda_\cF := \lambda_1 e_{F_1} +\lambda_2 e_{F_2\backslash F_1} +\dots +\lambda_ke_{F_k\backslash F_{k-1}} \, .
\]
The following summarizes results from \cite[Section 3.3(d)]{Fujishige:2005} and \cite[Proposition 1.3]{BilleraSarangarajan:1996}.
\begin{proposition} \label{faces-permutahedron}
	The map which sends the flag $\cF=(F_1,\dots,F_k)$ in $[n]$ to the non-empty proper face of $\Pi_n$ which maximizes $\lambda_\cF$ is a bijection.
	In particular,
	\begin{compactenum}
		\item the flag $\cF$ with $k$ constituents corresponds to a face of codimension $k$, which is affinely equivalent with $\Pi_{|F_1|} \times \Pi_{|F_2 \setminus F_1|} \times \dots \times \Pi_{|F_k \setminus F_{k-1}|} \times \Pi_{|[n] \setminus F_k|}$;
		\item the facets correspond to non-empty proper subsets of $[n]$;
		\item the vertices correspond to the full flags;
		\item the edges correspond to pairs of full flags which differ in exactly one constituent;
		\item each $2$-face, where $k=n-3$, is either a hexagon (if there exists $i$ with $|F_{i+1} \setminus F_i| = 3$), or it is a square (if there exist distinct $i,j$ with $|F_{i+1} \setminus F_i| = 2$ and $|F_{j+1} \setminus F_j| = 2$).
		\item for two sets $A,B \subset [n]$ with $|A| = |B|$ and $|A \triangle B| = 2$, the edges defined by the flags extending $(A \cap B, A \cup B)$ span a face isomorphic with $\Pi_{|A \cap B|} \times \Pi_{2} \times \Pi_{n - |A \cup B|}$.
	\end{compactenum}
\end{proposition}
Here, '$\triangle$' denotes symmetric difference. 
Notice that $\Pi_n$ also arises as a secondary polytope of the prism over an $(n{-}1)$-simplex; see \cite[Theorem 6.2.6]{Triangulations}.
This leads to another way to describe the face lattice.
The result below is known.
We indicate a proof for the sake of completeness.

\begin{proposition}\label{prop:group}
	For $n\geq 4$ the group of combinatorial automorphisms $\Aut(\Pi_n)$ is isomorphic to the direct product $\Sym(n)\times\ZZ_2$; its order equals $2\cdot n!$.
	Furthermore, $\Aut(\Pi_3)$ is the dihedral group of order $12$, and $\Aut(\Pi_2)$ is the cyclic group of order two.
	In all these cases, each combinatorial automorphism is linearly induced.
\end{proposition}

\begin{proof}
	The symmetric group $\Sym(n)$ operates on the vertices of $\Pi_n$ by multiplication on the right, and that operation is free.
        The facets of $\Pi_n$ correspond to subsets of size $k$ and whose complement has size $n-k$.
	The stabilizer of such a facet is a subgroup of $\Aut(\Pi_n)$ isomorphic to $\Sym(k)\times\Sym(n-k)$.
	It follows that the above operation of $\Sym(n)$ maps facets to facets, and thus $\Sym(n)\leq\Aut(\Pi_n)$.
	The polytope $\Pi_2$ is a segment, and $\Pi_3$ is a regular hexagon.
	For the rest of the proof we assume $n \geq 4$.
	
	We consider the linear form $\lambda(x)=x_1-x_2-x_{n-1}+x_n$ on $\RR^n$.
	This defines the reflection
	\[
	s(x) = x - \frac{x_1-x_2-x_{n-1}+x_n}{2}(1,-1,0,\dots,0,-1,1) \, .
	\]
	A direct inspection shows that $s$ permutes the vertices of $\Pi_n$, and it fixes the point $(1,2,\dots,n)$.
	It follows that the order of $\Aut(\Pi_n)$ is at least $2\cdot n!$.
	If $n\geq 4$ it can be checked that the reflection $s$ and $\Sym(n)$ commute in $\Aut(\Pi_n)$.
	For $n=3$ the linear form $\lambda$ reads $x_1-2x_2+x_3$, and the group $\langle \Sym(3), s\rangle\cong D_{12}$ is a semidirect product (but not a direct product).
	
	To finally show that $\langle \Sym(n), s\rangle=\Aut(\Pi_n)$ it suffices to show that the stabilizer of $o=(1,2,\dots,n)$ in $\Aut(\Pi_n)$ is generated by $s$.
	The vertex $o$ is incident with precisely $n-1$ facets, corresponding to the subsets $[1], [2], \dots, [n-1]$ of $[n]$, and the reflection $s$ exchanges $[k]$ with $[n-k]$.
	These two facets are isomorphic to $\Pi_k\times\Pi_{n-k}$, and they are the only ones with this property among the facets through $o$.
	The claim follows because each set of cardinality two only admits one nontrivial permutation.
	The stabilizer of $o$ and any one of its incident facets is trivial.
\end{proof}

\section{Flag matroids and subpermutahedra}


We collect various known results and combine them to yield a first few observations.

\subsection{Matroid polytopes and valuated matroids} A \emph{subpolytope} of a polytope is the convex hull of a subset of the vertices.
Each face is a subpolytope, but the converse is false.
Let $P$ be a generalized permutahedron.
A \emph{subpermutahedron} of $P$ is a subpolytope of $P$ which itself is a generalized permutahedron.
A \emph{matroid (base) polytope} $M$ of \emph{rank} $d$ is a subpermutahedron of the hypersimplex $\Delta(d,n)$. 
The bases $\cB(M) \subseteq \binom{[n]}{d}$ of a matroid consist of all subsets whose indicator vector is a vertex of $M$.
In the sequel, we will identify a matroid with its matroid base polytope. 
The \emph{uniform matroid} $U_{d,n}$, whose bases are exactly $\binom{[n]}{d}$, corresponds to $\Delta(d,n)$.

For a lattice polytope $P$, we abbreviate $P_\ZZ=P\cap\ZZ^n$.
A function $f:P_\ZZ\to\RR$ is \emph{M-convex} if the regular subdivision of the point configuration $P_\ZZ$ by the height function $f$ is permutahedral.
We use the convention that regular subdivisions are induced by lower convex hulls; see \cite{Triangulations} or \cite[\S1.2]{ETC}.
A \emph{valuated matroid} $\mu$ over a matroid $M$ is an M-convex function on $\cB(M)$. 
Equivalently, a function $\mu:\cB(M)\to \RR$ is a valuated matroid if the 3-term Pl\"ucker relations hold:
for each $S\in \binom{[n]}{d-2}$ and $i,j,k,l \notin S$, the minimum in
\begin{equation}\label{eq:3term}
  \min(\,\mu(Sij) + \mu(Skl),\, \mu(Sik) + \mu(Sjl),\, \mu(Sil) + \mu(Sjk) \,)
\end{equation}
is attained at least twice; see \cite[\S4.4]{Tropical+Book} and \cite[\S10.4]{ETC}.

We note that permutahedral subdivisions do not have to be regular, i.e., they do not have to arise from a height function. 
In particular, non-regular matroid subdivisions can be used to construct non-realizable oriented matroids~\cite{CelayaLohoYuen:2020a}. 

\subsection{Flag matroids}
Let $M$ and $N$ be matroids of ranks $d$ and $d+1$, respectively.
Then the pair $(M,N)$ forms a \emph{quotient} if the convex hull of $M\times\{1\}\cup N\times\{0\}$ is a matroid.
We denote this by $M \twoheadleftarrow N$.
Here and below we will identify matroids with their matroid base polytopes.
Two valuated matroids $\mu$ and $\nu$, with respective underlying matroids $M$ and $N$, are a \emph{(valuated matroid) quotient} if $M \twoheadleftarrow N$ and the \emph{3-term tropical incidence relations} are fulfilled; that is, for all $S\in\binom{[n]}{d-1}$ and $i,j,k \notin S$, 
\begin{equation} \label{eq:3-term-incidence} \tag{3TIR}
\min (\,\mu(Si) +\nu(Sjk),\; \mu(Sj) +\nu(Sik),\; \mu(Sk) +\nu(Sij) \,)
\end{equation}
attains the minimum at least twice.
Similarly, we denote this by $\mu \twoheadleftarrow \nu$.

We remark that (valuated) matroid quotients are often defined differently, but it is always equivalent to inclusion of (tropical) linear spaces. 
For consecutive ranks, this is equivalent to the existence of a (valuated) matroid $\xi$ over $[n+1]$ such that $\mu=\xi / (n+1)$ and $\nu = \xi \backslash (n+1)$.
Both equivalences can be found in~\cite{BrandtEurZhang:2021}. 
As the 3-term Pl\"ucker relations suffice to define valuated matroids, if its support is a matroid \cite[Corollary 5.5]{Rincon:2012}, our definition coincides with this.

A sequence of matroids $\cM = (M_{1},\ldots,M_{n})$ 
is a \emph{(full) flag matroid} if the rank of $M_i$ is $i$ and $M_{i}\twoheadleftarrow M_{j}$ for every $1\leq i<j\leq n$.
The \emph{flag matroid (base) polytope} of $\cM$ is 
\begin{equation*}
	M_1 + \dots + M_n = \conv\biggSetOf{\sum_{i=1}^{n}e_{B_i}}{(B_1,\dots,B_n) \in (\cB(M_{1}),\ldots,\cB(M_{n}))} \,.
\end{equation*}

The equivalence of (i) and (ii) in \cite[Theorem 1.11.1]{BorovikGelfandWhite:2003} characterizes base polytopes of flag matroids as those generalized permutahedra which arise as the convex hull over sums of characteristic vectors of flags of bases of a sequence of matroids.
We use a reformulation adapted from \cite[Theorem 4.1.5]{BrandtEurZhang:2021}.

\begin{theorem}[{Borovik, Gel$'$fand, and White \cite{BorovikGelfandWhite:2003}}] \label{thm:characterization-flag-base-polytope}
  A polytope is the base polytope of a full flag matroid if and only if it is a subpermutahedron of the regular permutahedron.
\end{theorem}

Similarly, a \emph{(full) valuated flag matroid} is a sequence of valuated matroids $(\mu_{1},\ldots,\mu_{n})$ such that $\mu_{i}\twoheadleftarrow \mu_{j}$ for every $1\leq i<j\leq n$.

The \emph{Dressian} $\Dr(d,n) \subset \RR^{\binom{n}{d}}$ is the set of all valuated matroids on the uniform matroid of rank $d$ on $n$ elements.
The \emph{flag Dressian} $\FlDr(n) \subset \RR^{2^n}$ is the set of all valuated flag matroids on the flag of uniform matroids. 
Both of them carry a natural fan structure where each cone comprises all (flag) valuated matroids which induce a given subdivision.

\subsection{Bruhat (interval) polytopes}
We recall the \emph{(strong) Bruhat order} of $\Sym(n)$.
Consider the simple transpositions $\tau_1,\dots,\tau_{n-1}$, with $\tau_i = (i,i+1)$, which generate the group $\Sym(n)$.
A sequence $\tau_{i_1},\dots,\tau_{i_\ell}$ of minimal length with $\sigma:=\tau_{i_1}\cdots\tau_{i_\ell}$, is a \emph{reduced word} of $\sigma$.
Now $\sigma_1 \leq \sigma_2$ if some reduced word of $\sigma_2$ contains a subsequence (not necessarily consecutive) which is a reduced word of $\sigma_1$.
This imposes a lattice structure on $\Sym(n)$ with rank function given by the lengths of reduced words.
For $\sigma_1 \leq \sigma_2$ the convex hull of the interval $[\sigma_1,\sigma_2]$ in the Bruhat order is a \emph{Bruhat (interval) polytope}. 

\begin{figure}[th]\centering
	\begin{tikzpicture}
		\draw (0,4) -- (2,3) -- (2,1) -- (0,0) -- cycle;
		\draw [red](0,4) -- (0,0) -- (-2,1) -- (-2,3) -- cycle;
		\node at (0,4.25) {$123$};
		\node at (0,-.25) {$321$};
		\node at (-2.25,3.25) {$213$};
		\node at (2.25,3.25) {$132$};
		\node at (2.25,.75) {$231$};
		\node at (-2.25,.75) {$312$};
		
		\tikzstyle{vertex} = [fill = white, draw=black]
		\filldraw[vertex] (0,4) circle (2.5pt);
		\filldraw[vertex] (2,3) circle (2.5pt);
		\filldraw[vertex] (2,1) circle (2.5pt);
		\filldraw[vertex] (0,0) circle (2.5pt);
		\filldraw[vertex] (-2,1) circle (2.5pt);
		\filldraw[vertex] (-2,3) circle (2.5pt);
	\end{tikzpicture}
	\qquad
	\begin{tikzpicture}
		\draw (0,4) -- (2,3) -- (2,1) -- (0,0) -- (-2,1) -- (-2,3) -- cycle;
		\draw[red] (-2,1) -- (2,3);
		\draw[blue] (-2,3) -- (2,1);
		\node at (0,4.25) {e};
		\node at (0,-.25) {$\tau_{1}\tau_{2}\tau_{1}=\tau_{2}\tau_{1}\tau_{2}$};
		\node at (-2.25,3.25) {$\tau_{1}$};
		\node at (2.25,3.25) {$\tau_{2}$};
		\node at (2.25,.75) {$\tau_{1}\tau_{2}$};
		\node at (-2.25,.75) {$\tau_{2}\tau_{1}$};
		
		\tikzstyle{vertex} = [fill = white, draw=black]
		\filldraw[vertex] (0,4) circle (2.5pt);
		\filldraw[vertex] (2,3) circle (2.5pt);
		\filldraw[vertex] (2,1) circle (2.5pt);
		\filldraw[vertex] (0,0) circle (2.5pt);
		\filldraw[vertex] (-2,1) circle (2.5pt);
		\filldraw[vertex] (-2,3) circle (2.5pt);
	\end{tikzpicture}
	\caption{Left: regular hexagon $\Pi_3$ subdivided into two flag polytopes, neither of which are Bruhat polytopes.  Right: Bruhat order of $\Sym(3)$.}
	\label{fig:bruhat}
\end{figure}
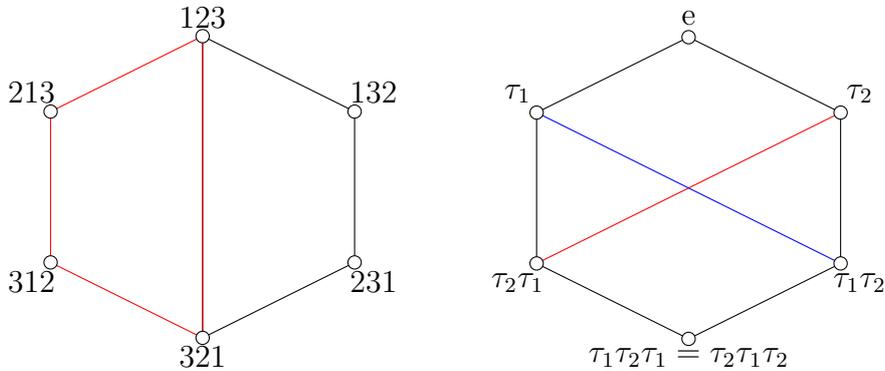

\begin{proposition}
  \label{prop:bruhat}
  A subpolytope of a permutahedron is a Bruhat interval polytope if and only if it is the base polytope of a flag matroid which can be realized in the totally non-negative flag variety.
\end{proposition}
\begin{proof}
  Any Bruhat interval $[\sigma_1,\sigma_2]$ defines a stratum $G^{>0}_{\sigma_1,\sigma_2}$ in the totally non-negative (TNN) flag variety and by \cite[Proposition 6.7 and Theorem 6.10]{KodamaWilliams:2015} the polytope of any flag associated to a point of $G^{>0}_{\sigma_1,\sigma_2}$ is the Bruhat polytope of $[\sigma_1,\sigma_2]$.
  Since the strata of the form $G^{>0}_{\sigma_1,\sigma_2}$ cover the TNN flag variety the converse also follows.
\end{proof}

This implies that every constituent of a Bruhat interval polytope is a positroid \cite[Corollary 6.11]{KodamaWilliams:2015}. However, the converse is not true:

\begin{example}	\label{ex:bruhat}
  Consider the quadrangle $\conv(123,213,312,321)$, which is a subpermutahedron of the hexagon $\Pi_3$; see Figure~\ref{fig:bruhat} (left).
  This is not a Bruhat interval polytope since the inclusion of the bottom element $123$ and the top element $321$ would imply that it is the entire permutahedron.
  Yet it is a flag matroid base polytope, with constituents $(\{1,3\},\{12,13,23\},\{123\})$. 
\end{example}

\subsection{Compression of valuated flag matroids}

A height function $w \colon \Sym(n) \to \RR$ is a \emph{valuated permutahedron} if each cell in the induced subdivision of $\Pi_n$ is a subpermutahedron. 
Alternatively, a height function $w \colon \Sym(n) \to \RR$ is a valuated permutahedron if and only if its piecewise-linear extension (i.e., linearly extending over each cell of the subdivision) on the lattice points $\Pi_n \cap \ZZ^{n}$ is an M-convex function.
In the latter case, it is the supremum over all M-convex functions that agree with $w$ on $\Sym(n)$.   

It turns out that these are essentially the same as full valuated flags of uniform matroids. 
Let $(\mu_1,\dots,\mu_k)$ be a sequence of valuated matroids with $\rk(\mu_1) < \dots < \rk(\mu_k)$ and underlying matroid $M_1,\dots,M_k$. 
In~\cite{FujishigeHirai:2022}, the \emph{compression} $u \colon \sum_{i=1}^{k}M_i \to\RR$ of $(\mu_1,\dots,\mu_k)$ is defined as the function
\begin{equation*}
	u(x)=\min\biggSetOf{\sum_{i\in[k]}\mu_i(Y_{i})}{x=\sum_{i\in[k]}e_{Y_{i}} \,\ Y_{i}\in \cB(M_i) \text{ for } i\in[k]} \; .
\end{equation*}

When $(\mu_1,\dots,\mu_k)$ in the former definition is a valuated flag matroid, by \cite[Theorem 4.4.2]{BrandtEurZhang:2021}, the cells in the subdivision of $\Pi_n \cap \ZZ^{n}$ induced by $u$ are themselves base polytopes of flag matroids.
Therefore, using \Cref{thm:characterization-flag-base-polytope}, the subdivision is composed of subpermutahedra of $\Pi_n$. 
On the other hand, \cite[Cor.~4.4.5]{BrandtEurZhang:2021} provides us with the converse, namely that such subdivisions arise from the compression of a valuated flag matroid. 

\begin{theorem}[{Brandt, Eur and Zhang \cite{BrandtEurZhang:2021}}]   
  \label{thm:equivalenc-permutahedra-flag}
  A height function $w \colon \Sym(n) \to \RR$ is a valuated permutahedron if and only if it is the restriction $u|_{\Sym(n)}$ of the compression of a full flag of valuated uniform matroids. 
\end{theorem}

More generally, \Cref{thm:characterization-flag-base-polytope} indicates that a (not-necessarily regular) permutahedral subdivision of a permutahedron can be seen as a \emph{flag matroid subdivision}. 

\subsection{Two canonical permutahedral subdivisions}

Here we describe special types of permutahedral subdivisions.

We start with the canonical subdivision of a flag matroid.
Each matroid arises as a cell in its \emph{corank subdivision} of the encompassing hypersimplex \cite{Speyer:2005}, which associates the corank to each set of rank-size.  
This motivates the consideration of the height function given by the compression of the flag of corank functions of the constituents in a flag matroid.  

\begin{proposition}
\label{prop:corank}
	Every subpermutahedron of $\Pi_n$ appears as a cell of a regular permutahedral subdivision.
\end{proposition}
\begin{proof}
	Given a subpermutahedron $P$, by \Cref{thm:characterization-flag-base-polytope}, it is the base polytope of a flag matroid $(M_1,\dots,M_n)$. 
	Let $\mu_i$ be the corank valuation of $M_i$, that is, $\mu_i(T) = \cork_{M_i}(T)$ for every $T \in \binom{[n]}{\rk(M_i)}$.  
	We show that $(\mu_1,\dots,\mu_n)$ is indeed a valuated flag matroid.
	
	For a fixed $d \in [n-1]$, we consider $M_d$ and $M_{d+1}$.
        Since $M_d$ is a quotient of $M_{d+1}$ there exists some matroid $N$ over $[n+1]$ such that $M/(n+1)= M_d$ and $M\backslash(n+1)= M_{d+1}$. 
	Notice that the corank valuation $\nu$ of $N$ restricts and contracts to $\mu_{d+1}$ and $\mu_d$ respectively. 
	That is, $\nu(T) = \cork_M(T) = \mu_{d+1}(T)$ if $n+1\notin T$ and  $\nu(T) = \cork_M(T) = \mu_d(T-(n+1))$ for $n+1\in T$.
        Since the corank valuation is always a valuated matroid \cite[Proposition 4.5.5]{Speyer:2005}, then $\nu$ is a valuated matroid and hence $\mu_d$ and $\mu_{d+1}$ satisfy the tropical incidence relations.
	So $(\mu_1,\dots,\mu_n)$ is a valuated flag matroid. 
	
	Its compression is zero on the vertices of $P$ and is positive on the other vertices of $\Pi_n$, so it induces a regular subdivision containing $P$ as a cell. 
\end{proof}

Now, we turn to the construction of a valuated full flag matroid canonically extending a valuated matroid. 
In particular, this gives rise to an embedding of the Dressian in the flag Dressian. 

Let $\mu \colon \binom{[n]}{d} \to \RR$ be a valuated matroid of rank $d$.
For $d \geq 1$, its \emph{truncation} is the function $\mu^{(1)} \colon \binom{[n]}{d-1} \to \RR$ with $\mu^{(1)}(S) = \min_{S \subset T} \mu(T)$.
Dually for $d \leq n-1$, its \emph{elongation} is the function $\mu^{(-1)} \colon \binom{[n]}{d+1} \to \RR$ with $\mu^{(-1)}(S) = \min_{S \supset T} \mu(T)$. 
In~\cite{Murota:1997}, these operations were shown to yield valuated matroids. 

Recall from~\cite[Definition 4.2.2]{BrandtEurZhang:2021}, that $\mu^{(1)} \twoheadleftarrow \mu$ if for arbitrary $I \in \binom{[n]}{d}, J \in \binom{[n]}{d+1}$ and $i \in I \setminus J$
\[
  \mu^{(1)}(I) + \mu(J) \geq \min_{j \in J \setminus I}\mu^{(1)}(I \setminus i \cup j) + \mu(J \setminus j \cup i) \;.
\]
Let $k \in [n] \setminus I$ be an element with $\mu^{(1)}(I) = \mu(I \cup k)$.
As $\mu$ is a valuated matroid, we get
\[
  \mu(I \cup k) + \mu(J) \geq \min_{j \in J \setminus (I \cup k)}\mu(I \cup k \setminus i \cup j) + \mu(J \setminus j \cup i) \;.
\]
By definition of truncation, we get $\mu(I \cup k \setminus i \cup j) \geq \mu^{(1)}(I \setminus i \cup j)$ for each $j \in J \setminus (I \cup k)$.
Hence,
\[
  \mu^{(1)}(I) + \mu(J) \geq \min_{j \in J \setminus (I \cup k)}\mu^{(1)}(I \setminus i \cup j) + \mu(J \setminus j \cup i) \geq \min_{j \in J \setminus I}\mu^{(1)}(I \setminus i \cup j) + \mu(J \setminus j \cup i) 
\]
which proves the claim. 
Dually, it follows that $\mu \twoheadleftarrow \mu^{(-1)}$.
We summarize this in the following statement. 

\begin{lemma} \label{lem:truncation+elongation+quotient}
  For a valuated matroid $\mu \colon \binom{[n]}{d} \to \RR$, its truncation and its elongation form a quotient with $\mu$, that is, we have $\mu^{(1)} \twoheadleftarrow \mu$ and $\mu \twoheadleftarrow \mu^{(-1)}$.
\end{lemma}

We define higher truncations and elongations by setting $\mu^{(k+1)} := \left(\mu^{(k)}\right)^{(1)}$ and $\mu^{(-k-1)} := \left(\mu^{(-k)}\right)^{(-1)}$ for $k \geq 0$ with $\mu^{(0)} := \mu$. 
Iterating this construction, Lemma~\ref{lem:truncation+elongation+quotient} shows that $(\mu^{(d-1)},\dots,\mu^{(1)},\mu,\mu^{(-1)},\dots,\mu^{(d-n)})$ is a flag matroid.
This gives rise to an embedding of the Dressian into the flag Dressian.

\begin{proposition}
  The map $\iota \colon \Dr \to \FlDr$ defined by
  \[
  \mu \mapsto \left(\mu^{(d-1)},\dots,\mu^{(1)},\mu,\mu^{(-1)},\dots,\mu^{(d-n)}\right) 
  \]
  is a piecewise-linear embedding of the Dressian in the flag Dressian. 
\end{proposition}

\subsection{Incidence relations imply Pl\"ucker relations}

Now, we deal with the interplay of the valuated matroids in a flag.
For (non-tropical) Pl\"ucker vectors, the implication of the Pl\"ucker relation from the incidence relation occurred in~\cite{JellMarkwigRinconSchroter:2003.02660}.
On the combinatorial level, this was studied in the context of \enquote{M$^{\natural}$-convex set functions}; cf.\ \cite{MurotaShioura:2018}.
Indeed, imposing supermodularity among the constituents of a valuated flag matroid gives rise to an M$^{\natural}$-convex set function. 

\begin{theorem} 
  \label{thm:1step}
  Let $\mu \colon \binom{[n]}{d} \to \RR$ and $\nu \colon \binom{[n]}{d+1} \to \RR$ be any two functions satisfying the tropical incidence relations \eqref{eq:3-term-incidence}. 
  Then $\mu$ and $\nu$ are valuated matroids.
\end{theorem}
\begin{proof}
	We show that $\nu$ satisfies the 3-term Pl\"ucker relations \eqref{eq:3term}.
	For a contradiction, suppose there is a set $S$ and $i,j,k,l \in [n]\backslash S$ with
	\begin{equation} \label{eq:violated-twice-min}
		\nu(Sij) + \nu(Skl) < \nu(Sik) + \nu(Sjl) \quad \text{and} \quad \nu(Sij) + \nu(Skl) < \nu(Sil) + \nu(Sjk) \, .
	\end{equation}
	
	Let $\xi = \mu(Si) e_i + \mu(Sj) e_j + \mu(Sk) e_k + \mu(Sl) e_l$.
	Defining $\nu'$ with $\nu'(T) = \nu(T) - \langle\xi,e_T\rangle$ for all $T \in \binom{[n]}{d+1}$ and defining $\mu'$ by the same translation from $\mu$, the pair $(\mu,\nu)$ is a valuated matroid quotient if and only if $(\mu',\nu')$ is a quotient.
	Hence, we can assume that $\mu(Si) = \mu(Sj) =\mu(Sk) = \mu(Sl) = 0$. 
	With this, applying \eqref{eq:3-term-incidence} to the index triplet $ijk$ yields that the minimum in $\min(\nu(Sij),\nu(Sik),\nu(Sjk))$ is attained twice.
        Similarly, the two minima $\min(\nu(Sik),\nu(Sil),\nu(Skl))$, from $ikl$, and $\min(\nu(Sij),\nu(Sil),\nu(Sjl))$, from $ijl$, are attained twice.
        
        By \eqref{eq:violated-twice-min}, $\min(\nu(Sij),\nu(Skl)) \leq \max(\nu(Sik),\nu(Sjl),\nu(Sil),\nu(Sjk))$, so we can assume that $\nu(Sij) < \nu(Sik)$.
	Combining this with the minimum condition for $ijk$ yields $\nu(Sij) = \nu(Sjk)$.
	From the second inequality in~\eqref{eq:violated-twice-min}, we get $\nu(Skl) < \nu(Sil)$.
	By the minimum condition, we have $\nu(Skl) = \nu(Sik)$. 
	From the first inequality in~\eqref{eq:violated-twice-min}, we get $\nu(Sij) < \nu(Sjl)$.
	Again by the minimum condition, $\nu(Sij) = \nu(Sil)$. 
	With the above, this yields $\nu(Sij) = \nu(Sil) > \nu(Skl) = \nu(Sik)$, which contradicts the original assumption. 
	Hence $\nu$ satisfies all 3-term Pl\"ucker relations and so it is a valuated matroid.

	By duality, $\mu$ must also be a valuated matroid.
\end{proof}

We will use the following example as a coherent illustration of the main theorems.
\begin{example}\label{ex.1} 
Let $\puiseuxseries{\RR}{t}$ denote the field of Puiseux series with the valuation $\val \colon \puiseuxseries{\RR}{t}\to\QQ$ such that $\val(q)=k$, where $k$ is the lowest exponent that appears in $q$. 
We begin with the $3{\times}3$ matrix with entries from $\puiseuxseries{\RR}{t}$ below.
\begin{center}
	$A=\begin{bmatrix}
		t & 1 &t^{2}\\
		t^{4} & 1+t^{2} & t\\
		1&1&1
	\end{bmatrix}$
\end{center}
Let $L_{i}$ be the span of the first $i$ rows of $A$. 
Below we have the Pl\"ucker coordinates of $L_{1},L_{2}$, and $L_{3}$. 
Note that the entries of $p_{L_{i}}$ are indexed by the elements of ${[3]\choose i}$, ordered lexicographically.
We have
\[
    p_{L_{1}}=(t,1,t^{2}),\; p_{L_{2}}=(t+t^{3}-t^{4},t^{2}-t^{6},t-t^{2}-t^{4}),\; p_{L_{3}}=2t-2t^2+t^{3}-2t^{4}+t^{6}
\]

\begin{figure}[th]\centering
  \begin{tikzpicture}
    \draw (0,4) -- (2,3) -- (2,1) -- (0,0) -- (-2,1) -- (-2,3) -- cycle;
    \draw (-2,1) -- (2,3);
    \node at (.25,4.25) {$w((1,2,3))=4$};
    \node at (.25,-.30) {$w((3,2,1))=3$};
    \node at (-3.30,3.30) {$5=w((2,1,3))$};
    \node at (3.25,3.25) {$w((1,3,2))=2$};$
    \node at (3.25,.7) {$w((2,3,1))=2$};
    \node at (-3.25,.65) {$4=w((3,1,2))$};$
    
    \tikzstyle{vertex} = [fill = white, draw=black]
    \filldraw[vertex] (0,4) circle (2.5pt);
    \filldraw[vertex] (2,3) circle (2.5pt);
    \filldraw[vertex] (2,1) circle (2.5pt);
    \filldraw[vertex] (0,0) circle (2.5pt);
    \filldraw[vertex] (-2,1) circle (2.5pt);
    \filldraw[vertex] (-2,3) circle (2.5pt);
  \end{tikzpicture}
  \caption{Subdivision of $\Pi_3$ into Bruhat interval polytopes.}
  \label{fig:encompassing}
\end{figure}
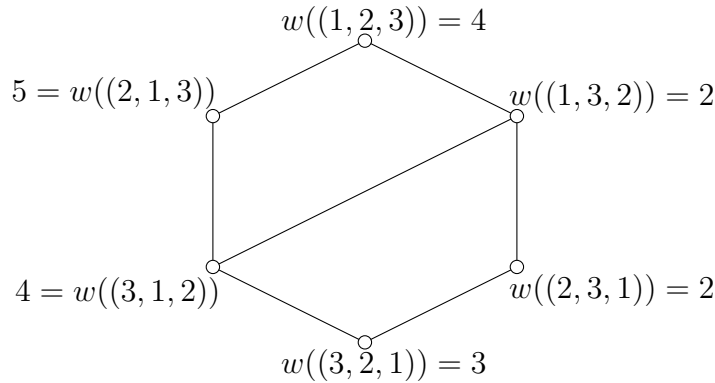

Tropicalizing the Pl\"ucker coordinates of $L_{1}$ and $L_{2}$ and $L_{3}$ yields functions $\mu_{i}:{[3]\choose i}\to\RR$, $i\in[3]$, such that
\[
    \mu_{1} = (1,0,2),\;
    \mu_{2} = (1,2,1),\;
    \mu_{3}=1 \enspace .
\]

The Dressian $\Dr(d,n)$ is known to contain the tropicalized Grassmannian $\TGr(d,n)$ as a set.
Therefore we know $\mu_1,\mu_2,\mu_3$ to be valuated matroids. 
As we have $L_1\subset L_2\subset L_3$, the sequence $(\mu_1,\mu_2,\mu_3)$ is a valuated full uniform flag matroid on $[3]$.
The restriction of the compression to the vertices yields a permutahedral subdivision of $\Pi_3$; see \Cref{fig:encompassing}.
\end{example}

\section{Regular permutahedral subdivisions}

Based on the structure of the permutahedron, we derive conditions for a regular subdivision to be permutahedral: 
a height function induces a permutahedral subdivision of $\Pi_n$ if and only if it does so in the 2-skeleton of $\Pi_n$.

To this end consider an arbitrary polytope $P$ with vertex-edge graph $\Gamma=(V,E)$.
We duplicate each edge, equipped with two opposite orientations; this turns $\Gamma$ into a directed graph, which we denote $\Gamma^\pm=(V,E^\pm)$.
Now any function $f:V\to\RR$ defined on the vertices yields a function $g:E^\pm\to\RR$ on the directed edges by letting
\begin{equation}\label{eq:potential}
  g(v,w) = f(v)-f(w) \qquad \text{for distinct } v,w\in V \; ,
\end{equation}
where $g(w,v)=-g(v,w)$.
We assume $n:=\dim P\geq 2$, whence $\partial P\approx\Sph^{n-1}$ is simply connected.
The function $f$ can be recovered from $g$ under the following conditions:
\begin{proposition}\label{prop:potential}
  Let $g:E^\pm\to\RR$ be a function on the directed edges of $P$ which satisfies $\sum_{i=1}^{k} g(v_i,v_{i+1})=0$ for every $2$-face of $P$ with vertices $v_1,v_2,\dots,v_{k+1}=v_1$ (labeled cyclically).
  Then there is a unique function $f:V\to\RR$ on the vertices with \eqref{eq:potential} and $f(s)=f_0$ for any fixed $s\in V$ and $f_0\in\RR$.
\end{proposition}
\begin{proof}
  Pick a directed spanning tree $T$ of $\Gamma^\pm$ rooted at $s$, and define a function $f:V\to\RR$ by inductively setting $f(s)=f_0$ and $f(v) = g(v,w) + f(w)$ along the directed edges of $T$.
  We need to show that $f$ satisfies $f(v)-f(w)=g(v,w)$ for distinct vertices $v,w$.
  To this end it suffices to prove that $h(c):=\sum_{i=1}^{\ell} g(w_i,w_{i+1})=0$ holds for any directed cycle $c=(w_1,w_2,\dots,w_\ell,w_{\ell+1}=w_1)$ in $\Gamma$.
  The boundary complex $\partial P$ is a polytopal complex homeomorphic to $\Sph^{n-1}$.
  A combinatorial procedure for computing the fundamental group of a polytopal complex is given in \cite[\S44]{SeifertThrelfall} (where this is proved for simplicial complexes).
  This has the following consequences.
  First, any path in $\partial P$ is homotopic to a path in $\Gamma$.
  Second, due to $n\geq 2$, the boundary $\partial P \approx \Sph^{n-1}$ is simply connected, and thus every closed path in $\Gamma$ can be contracted to a constant path within the $2$-skeleton of $P$.
  In this way, up to homotopy, the cycle $c$ can be contracted combinatorially in the following sense: there is a sequence of cycles $c_1, c_2, \dots, c_m$ in $\Gamma$ such that $c_1=c$, the cycle  $c_m$ is trivial (without any edges), and the symmetric difference between the edges in $c_i$ and $c_{i+1}$ forms a $2$-face.
  Now the assumption on $g$ gives $0=h(c_m)=h(c_{m-1})=\dots=h(c_0)=h(c)$, as we wanted to show.
  A similar argument yields the uniqueness of $f$.
\end{proof}
Before we will apply this statement to the hypersimplex, we need to relate functions on the permutahedron and on the hypersimplex. 
To this end let $G^\pm_{\Pi}(n)$ be the directed vertex-edge graph of $\Pi_n$ as above.
Similarly, let $G^\pm_\Delta(d,n)$ be a directed version of the vertex-edge graph of the hypersimplex $\Delta(d,n)$.
Recall that each edge of $\Pi_n$ gives rise to a pair $(A,B)$ of equicardinal subsets of $[n]$ with $|A \triangle B| = 2$, and that each edge of $\Delta(d,n)$ corresponds to a pair $(A,B)$ of d-subsets of $[n]$ with $|A \triangle B| = 2$. 

\begin{lemma} \label{lem:squares-function}
  Let $g$ be a function on the directed edges of $G^\pm_{\Pi}(n)$ such that, on each $2$-face of $\Pi_n$ which is a square, parallel directed edges attain the same $g$-value.
  Then this yields a function $g'$ on the directed edges of $\bigcup_{d=0}^{n} G^\pm_{\Delta}(d,n)$ such that $g'(A,B) = g(e)$ for all equicardinal subsets of $[n]$ with $|A \triangle B| = 2$ and all edges $e$ of $\Pi_n$ corresponding to $(A,B)$. 
\end{lemma}
\begin{proof}
  Let $(A,B)$ be a pair of equicardinal subsets of $[n]$ with $|A \triangle B| = 2$.
  By \Cref{faces-permutahedron}, the vertices given by the full flags extending $(A \cap B, A \cup B)$ form a face isomorphic with $\Pi_{|A \cap B|} \times \Pi_{2} \times \Pi_{n - |A \cup B|}$.
	In particular, any two of these edges are connected by a sequence of squares on this face, where each square is composed by a pair of parallel edges corresponding to $(A,B)$ and are adjacent to each other in this sequence along these edges.
  By the condition on the squares, all the edges in $G^\pm_{\Pi}(n)$ corresponding to $(A,B)$ have the same $g$-value.
  So we take that number to define $g'$ on the arc of $G^\pm_{\Delta}(d,n)$ corresponding to $(A,B)$, where $d = |A| = |B|$.
\end{proof}

The following is our first main result.

\begin{theorem}
	\label{thm:regular-2-skeleton}
	A height function $w \colon \Sym(n) \to \RR$ induces a permutahedral subdivision if and only if it induces a permutahedral subdivision of the 2-skeleton of $\Pi_n$. That is:
	\begin{enumerate}[leftmargin=*, widest=XXXX]
		\item[(HEX)] for every hexagon $abcdef$ (labeled cyclically) in the 2-skeleton of $\Pi_n$, we have
		\begin{enumerate}
			\item[(HXE)] $w(a)+w(c)+w(e) = w(b)+w(d)+w(f)$, 
			\item[(HXM)] the maximum in $\max(w(a)+w(d), w(b)+w(e), w(c)+w(f))$ is attained twice;
		\end{enumerate}
		\item[(SQR)] for every square face $abcd$ of $\Pi_n$ (labeled cyclically): $w(a)+w(c) = w(b)+w(d)$.
	\end{enumerate}
\end{theorem}

\begin{proof}
	A subdivision of a polytope defines a subdivision on each face.
	Therefore, as each cell in the subdivision is a subpermutahedron, the 2-faces are also subdivided into subpermutahedra. 
	Since squares cannot be subdivided any further, the second condition follows.
	To show the first condition, observe that at least four cyclically consecutive vertices have to lie in the same hyperplane, the other two on or above. 
	By relabeling and subtracting a linear form, we can assume that $w(a) = w(b) = w(c) = w(d) = 0$ and that $w(e) = w(f) \geq 0$. 
	This implies that the conditions on the $2$-skeleton are necessary.
	
	For the converse, suppose that $w$ satisfies the conditions (HEX) and (SQR).
	We will show that $w$ can be decomposed into a flag of valuated matroids $(g_1,\dots,g_n)$ such that $w$ is the result of their compression. 
	Then, by \Cref{thm:equivalenc-permutahedra-flag}, $w$ is a valuated permutahedron.
	
	On each directed edge $(a,b)$ in $G^\pm_\Pi(n)$, we let $h'(a,b) = w(a) - w(b)$ as above.
	Using (SQR), by \Cref{lem:squares-function}, this defines a function $h_d$ on the directed edges of $G^\pm_{\Delta}(d,n)$ for any $d$.
	We fix an arbitrary vertex $u$ of $\Pi_n$, which corresponds to a full flag $\mathcal{G} = (G_1, \dots, G_n)$ in $[n]$.
	Recall that a hexagon corresponds to the flags extending $(S,S {} ijk)$ for some $S$ and $i,j,k \in [n] \setminus S$.
	With this notation the condition (HXE) amounts to
	\[
	h(S {} i, S {} j) + h(S {} j, S {} k) + h(S {} k, S {} i) =
	h(S {} ij, S {} jk) + h(S {} jk, S {} ik) + h(S {} ik, S {} ij) = 0 \,.
	\]
	It follows that $h$ sums to zero along each oriented $2$-face of $\Delta(d,n)$ (i.e., a triangle).
        Thus \Cref{prop:potential} yields a function $g$ on the vertices of $\Delta(d,n)$ with $\sum_{d=1}^{n}g(G_d) = w(u)$.
	
\begin{figure}[th]
  \centering
  \includegraphics[height=.25\textwidth]{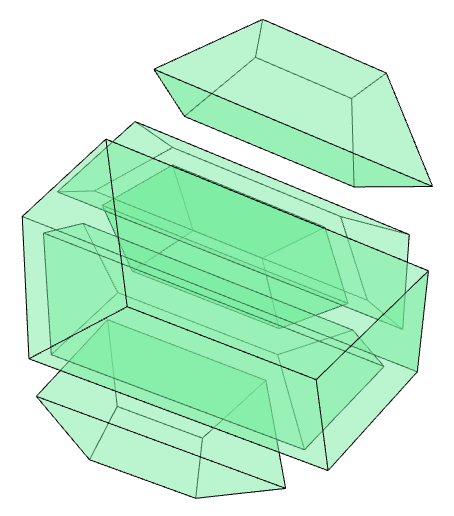}\qquad
  \includegraphics[height=.25\textwidth]{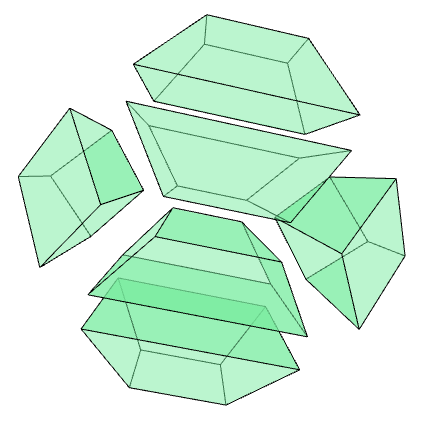}
  \caption{Two permutahedral subdivisions inducing the same pattern on the 2-skeleton.}
  \label{fig:2-skeleton}
\end{figure}

	Let $v$ be an arbitrary vertex of $\Pi_n$, which corresponds to a full flag $\cF = (F_1, \dots, F_n)$ in $[n]$.
	As a path from $u$ to $v$ decomposes into paths from $G_d$ to $F_d$ on $\Delta(d,n)$, we have $w(v)=\sum_{d = 1}^{n} g_d(F_d)$. 
	For a hexagon face described by the flags extending $(S,S {} ijk)$ with some $S \subset [n]$ and $i,j,k \in [n] \setminus S$, the sum representation of $w$ agrees in all terms with $d \leq |S|$ and $d \geq |S| + 3$.
	Hence, by (HXM), the maximum is attained at least twice in 
	\begin{equation*}
		\begin{aligned}
			\max\bigl(\, g(S {} i) + g(S {} ij) + g(S {} k) + g(S {} jk),\; & g(S {} j) + g(S {} jk) + g(S {} i) + g(S {} ik), \\
			&g(S {} k) + g(S {} ik) + g(S {} j) + g(S {} ij) \, \bigr) \;.
		\end{aligned}
	\end{equation*}
	
	Subtracting $g(S {} i) + g(S {} j) + g(S {} k) + g(S {} ij) + g(S {} ik) + g(S {} jk)$ and multiplying by $-1$ yields that $\min(g(S {} j) + g(S {} ik), g(S {} k) + g(S {} ij), g(S {} i) + g(S {} jk))$
	attains the minimum twice.
	This is the 3-term incidence relation.
	Summarizing, we have proven that $(g_1,\dots,g_n)$ satisfy the 1-step Pl\"ucker relations.
	By \Cref{thm:1step} all of the $g_d$ are valuated matroids and so $(g_1,\dots,g_n)$ is a valuated flag matroid.
\end{proof}

\begin{example}\label{ex.2}
  We continue \Cref{ex.1}.
  Define $w:\Sym(3)\to \RR$ to be the compression of $(\mu_{1},\mu_{2},\mu_{3})$ restricted to the vertices of $\Pi_{3}$.
One can check that
\[ w((1,2,3))+w((3,1,2))+w((2,3,1))=w((2,1,3))+w((1,3,2))+ w((3,2,1))=10 \]
and
\[ \max\left\{w((1,2,3))+w((3,2,1)),w((2,1,3))+w((2,3,1)),w((1,3,2))+w((3,1,2))\right\}=7 \]
is attained twice.
That is, $w$ meets the criteria described in \Cref{thm:regular-2-skeleton} to induce permutahedral subdivisions on hexagons.
\end{example}

\section{Explicit computations}
\label{sec:computations}
The secondary fan of a point configuration stratifies the height functions by the combinatorial type of the regular subdivisions induced.
By taking closures and intersecting with the unit sphere (in the space of height functions) this yields a decomposition of that sphere into spherical polytopes.
Restricting to permutahedral subdivisions yields a subfan and spherical polytopal subcomplex.
\begin{example}
  We sketch the situation for the regular hexagon $\Pi_3$.
  Its secondary fan, $\Sigma$, is a complete simplicial fan which is $3$-dimensional, modulo linealities.
  In fact, it is the normal fan of a $3$-dimensional associahedron.
  The 14 maximal cones $\sigma$ correspond to the vertices of the associahedron and also to the labeled triangulations of $\Pi_3$; cf.\ \cite[Figure~1.15]{Triangulations}.
  There are exactly three nontrivial permutahedral subdivisions of $\Pi_3$, which arise from three splits through opposite pairs of vertices of $\Pi_3$; cf.\ \cite{HerrmannJoswig:2008}.
  Each such split has a $1$-dimensional secondary cone, i.e., it forms a ray of $\Sigma$.
  Consequently, the spherical polytopal complex of permutahedral subdivisions of $\Pi_3$ comprises three isolated points in $\Sph^2$; see Figure~\ref{fig:3points}.
\end{example}

\begin{figure}[th]
  \centering
  \includegraphics[width=.2\textwidth]{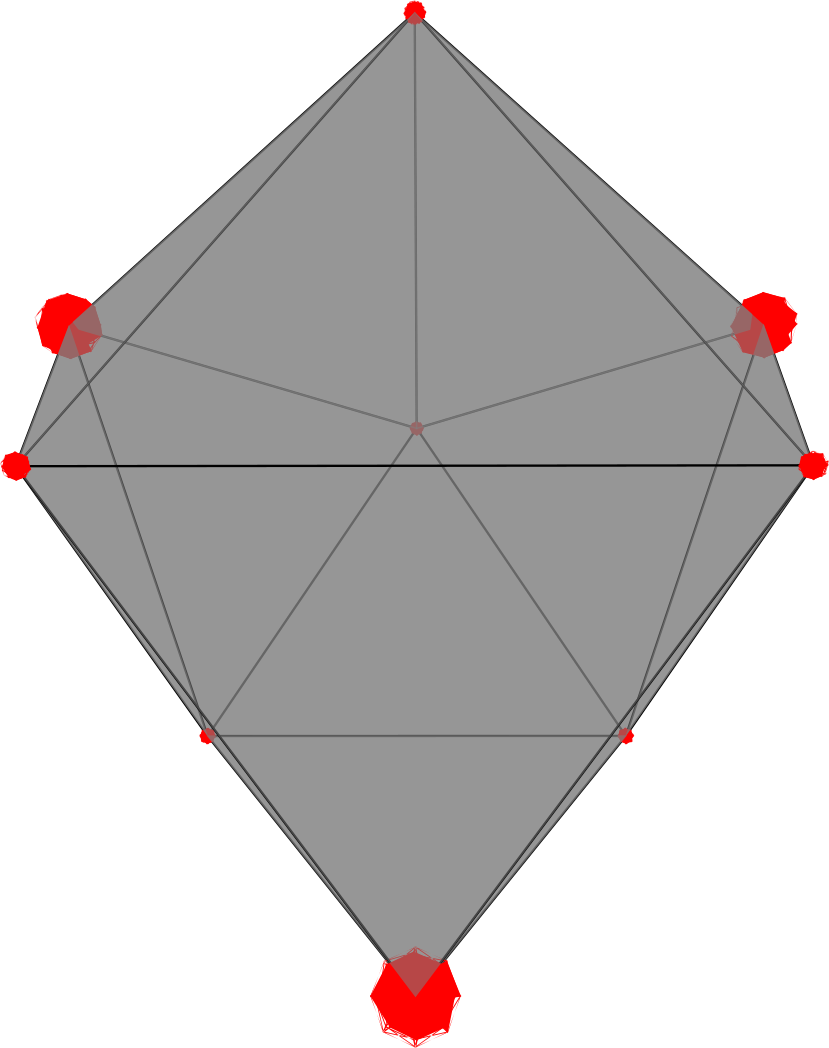}
  \caption{Locating the three points corresponding to the permutahedral subdivisions of $\Pi_3$ in the dual of the $3$-dimensional associahedron.}
\label{fig:3points}
\end{figure}

Tropicalizations of flag varieties for $n=4$ and $n=5$ have been computed by Bossinger, Lamboglia, Mincheva and Mohammadi \cite{BLMM:2017}.
Here we use the (HEX) and (SQR) conditions to compute the height vectors in $\RR^{24}$ that induce permutahedral subdivisions of the three dimensional regular permutahedron $\Pi_4$, employing \polymake \cite{DMV:polymake}.
That is, we intersect the eight tropical hypersurfaces corresponding to hexagonal facets with fourteen linear subspaces corresponding to hexagons and squares in the two-skeleton of $\Pi_4$.
We obtain a six-dimensional polyhedral fan, which we denote $\Phi_4$, and which has a three-dimensional lineality subspace.
The $f$-vector of the fan $\Phi_4$ reads $(20,76,75)$.
Out of the $75$ maximal cones $72$ are simplicial cones, i.e., they are spanned by three rays. 
The remaining three cones are spanned by four rays each.
In the secondary fan of $\Pi_4$, such a cone would be triangulated into two distinct maximal cones, $C_1$ and $C_2$.
The subdivisions of $\Pi_4$ corresponding to $C_1$ and $C_2$ are distinct, however, they induce the same permutahedral subdivision on the $2$-skeleton, \Cref{fig:2-skeleton}.
In other words, the subfan from the secondary fan of $\Pi_n$ corresponding to permutahedral subdivisions is not the same as the fan induced by (HEX) and (SQR); the latter is coarser. 
This contrasts with \cite[Theorem 14]{OPS:2018} where the Dressian has the same fan structure whether taken from the Pl\"ucker relations or as a subfan of the secondary fan of the matroid polytope.

Observe that the secondary fan refinement of $\Phi_4$ has $72 + 2\cdot 3 = 78$ maximal cones, which is the number reported in \cite{BLMM:2017}.
We summarize our findings as follows; see \Cref{fig:flag-dressian-4} for a visualization.
The fundamental group was computed with \OSCAR v0.9.0 \cite{OSCAR,OSCAR-book}.

\begin{proposition}
  The link of the fan $\Phi_4$ is a pure $2$-dimensional spherical polytopal complex with $f$-vector $(20,76,75)$.
  That complex is simply connected with top integral homology isomorphic to $\ZZ^{18}$.
\end{proposition}

\begin{figure}[th]
  \centering
  \includegraphics[width=.45\textwidth]{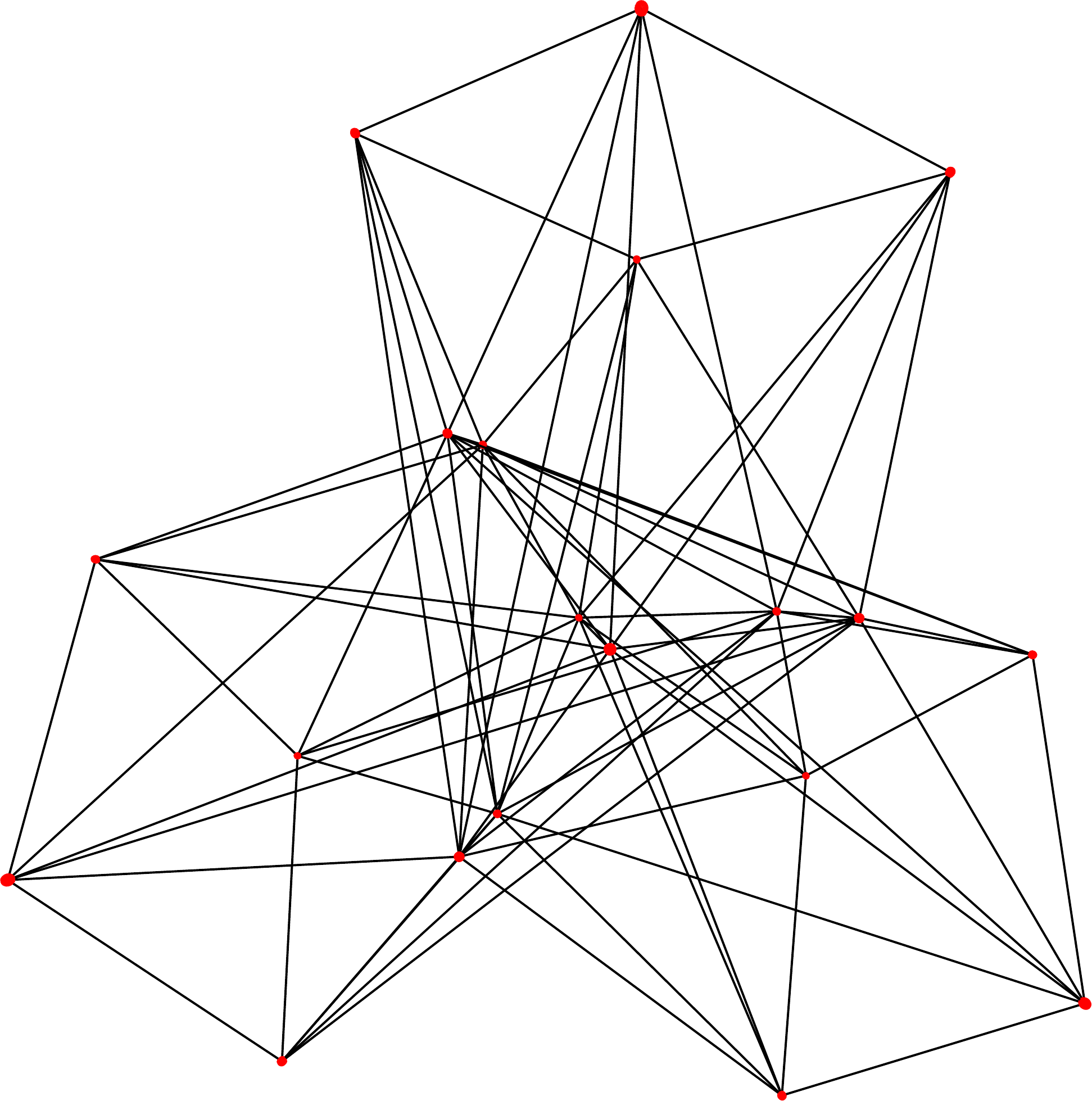}
  \caption{Graph of the link of the fan $\Phi_4$.  The three quadrangles sit at the \enquote{outside}.}
\label{fig:flag-dressian-4}
\end{figure}

Using the (HEX) and (SQR) conditions alone to compute the entire collection of points inducing permutahedral subdivisions of $\Pi_n$ seems to be intractable for $n>4$.

\section{Total positivity} 
The positive tropical Grassmannian $\TGr^+(d,n)$ is the tropicalization of the positive part of the Grassmannian $\Gr^+(d,n)$ consisting of linear spaces over the reals with all Pl\"ucker coordinates positive; see \cite{SpeyerWilliams:2005,SpeyerWilliams:2021,ArkaniHamedLamSpradlin:2021}.
Following~\cite[Equation~(1.1)]{ArkaniHamedLamSpradlin:2021}, a valuated matroid $v$ is \emph{positive} if it fulfills the three-term positive tropical Pl\"ucker relations; i.e., for every $S$ and $i < j < k < l$ not in $S$, we have
\begin{equation} \tag{3TPR+} \label{eq:positivepluecker}
	v(Sik) + v(Sjl) = \min(v(Sij)+v(Skl), v(Sil)+v(Sjk)) \;.
\end{equation}

The subsets of $\Dr(d,n)$ and of $\FlDr(n)$ for which all valuated matroids fulfill \eqref{eq:positivepluecker} are the \emph{positive Dressian} and the \emph{positive flag Dressian}, respectively. 

Similarly, the positive part of the (full) flag variety can be defined as the space of full flags $\cL=(L_1, \dots, L_n)$ where $L_d \in \Gr^+(d,n)$.
Let us formulate an \enquote{untropicalized} analog of one of the directions of \Cref{thm:regular-2-skeleton}.
The Pl\"ucker coordinates of a linear space $L_d\subseteq \KK^n $ over some field $\KK$ can be organized into a polynomial: 
\[
f_{L_d}(x) = \sum_{B\in \binom{[n]}{d}} p_B \prod_{i\in B} x_i \in \KK[x_1,\dots,x_n] \;,
\]
whose support is the base polytope of the matroid $M(L_d)$.
As a side remark, note that in~\cite{Purbhoo:2018} such a polynomial was shown to encode a totally nonnegative linear space if and only if it is stable.
Given the flag $\cL$, the coefficients of these polynomials satisfy
\begin{equation} \label{eq:classical-incidence}
p_{Si}p_{Sjk} - p_{Sj}p_{Sik} + p_{Sk}p_{Sij} = 0 
\end{equation}
for all $S\subseteq [n]$ and $i<j<k$ not in $S$.
To give an analog of \eqref{eq:positivepluecker}, we work our way from the equation above to reach the condition (HXM).

The product $f_{L_1}\cdots f_{L_n}$ is a polynomial whose support is the flag matroid base polytope $M(L_1) + \dots+ M(L_n)$. 
Suppose this polytope is the regular permutahedron $\Pi_n$ and let $q_\sigma$ be the coefficient of $x_1^{\sigma(1)}\cdots x_n^{\sigma(n)}$ in $f_{L_1}\cdots f_{L_n}$.
Hence, for $n=3$, we have that $q_{\sigma} = p_{\sigma^{-1}(3)}p_{\sigma^{-1}(2)\sigma^{-1}(3)}p_{\sigma^{-1}(1)\sigma^{-1}(2)\sigma^{-1}(3)}$ yielding with~\eqref{eq:classical-incidence} then 
\[
\begin{aligned}
	q_{(321)}q_{(123)}&q_{(231)}q_{(213)} - q_{(231)}q_{(213)}q_{(312)}q_{(132)} +  q_{(312)}q_{(132)}q_{(321)}q_{(123)} \\
	&= p_1p_2p_3p_{12}p_{13}p_{23}p_{123}^4 \left(p_{12}p_3-p_{13}p_2+p_{23}p_1 \right) = 0 \;.
\end{aligned}
\]
Notice as well that $p_1p_2p_3p_{12}p_{13}p_{23} = q_{(123)}q_{(231)}q_{(312)}=q_{(321)}q_{(213)}q_{(132)}$.
From here we can deduce the following relations that must be satisfied for any $n$:
\begin{itemize}
	\item For every hexagonal face $abcdef$ (labeled cyclically) of $\Pi_n$ we have $q_aq_cq_e = q_bq_dq_f$ and $q_bq_cq_eq_f -q_aq_cq_bq_f+ q_aq_bq_dq_e = 0$. By \cite[Section 4.3]{TsukermanWilliams} hexagons are Bruhat interval polytopes, and the negative term  $q_aq_cq_bq_f$ is the one not containing the lowest and largest elements of this interval.
	\item For every square face $abcd$ (labeled cyclically) of $\Pi_n$ we have $q_aq_c = q_bq_d$.
\end{itemize}

%

This suggests the following positivity condition for the tropicalization of the positive part of the flag variety; see \Cref{ex:bruhat}: 
\begin{enumerate}[leftmargin=*, widest=XXXXX]
	\item[(HXM+)] \label{eq:positiveflag} For every (cyclically labeled) hexagon $abcdef$ of $\Pi_n$, where $b$ is the lowest permutation in the Bruhat order, $w(b)+w(e) = \max(w(a)+w(d), w(c)+w(f))$.
\end{enumerate}

We will need the following positivity adaptation of \Cref{thm:1step}:
\begin{lemma}\label{lemma:1step_positive}
	Let $\nu : \binom{[n]}{d} \to \RR$ and $\mu : \binom{[n]}{d+1} \to \RR$ be any two functions such that for every $S\in \binom{[n]}{d-1}$ and $i<j<k \notin S$, 
	\[
	\nu(Sj)+\mu(Sik) = \min(\nu(Si)+\mu(Sjk),\nu(Sk)+\mu(Sij)) \;.
	\]
	Then $\nu\in \TGr^+(d,n)$ and $\mu \in  \TGr^+(d+1,n)$.
\end{lemma}
\begin{proof}
	By \Cref{thm:1step} we know already that $\nu$ and $\mu$ are valuated matroids. 
	Again we can do a translation to assume $\nu(Si)=\nu(Sj)=\nu(Sk)=\nu(Sl)=0$.
	Suppose $\mu$ does not satisfy \eqref{eq:positivepluecker}.
	Then w.l.o.g.~$\mu(Sik)+\mu(Sjl)>\mu(Sij)+\mu(Skl)$. 
	However, by the assumption of the lemma, $\mu(Sik)\leq \mu(Sij)$ and $\mu(Sjl) \leq \mu(Skl)$, a contradiction. 
\end{proof}

The following is a crucial step in our analysis.
\begin{lemma} \label{lem:positive-valuated-flag-matroid}
  Let $(w_1,\dots, w_n)$ be a valuated flag matroid whose compression fulfills (HXM+).
  Then it can be realized by a totally positive flag of linear spaces.
\end{lemma}
\begin{proof}
	Consider the valuated matroid $\mu$ on the uniform matroid $U_{n,2n}$ given by $\mu(B) = w_{|B\cap[n]|}(B\cap[n])$; this construction appears, e.g., in \cite{MurotaShioura-conjugate:2018}. 
	That $\mu$ is actually a valuated matroid depends on the right choice of $w_d$, which vary up to adding a constant. 
	To see this and that, moreover, $\mu \in \TGr^+(n,2n)$ we look at the 3-term Pl\"ucker relations \eqref{eq:positivepluecker}. 
	These are given by a set $S\in \binom{[2n]}{d-2}$ and $i,j,k,l \notin S$ and suppose $i< j<k< l$.
	We have the following cases:
	\begin{itemize}
		\item $|\{i,j,k,l\}\cap[n]| \le 1$. In this case all terms in the Pl\"ucker relation are equal.
		\item $|\{i,j,k,l\}\cap[n]| = 2$. Assuming $i,j \in [n]$, we have that $\mu(Sik) +\mu(Sjl)= \mu(Sil)+\mu(Sjk) = w_{m+1}(Ti) + w_{m+1}(Tj)$ and $\mu(Sij)+\mu(Skl) = w_{m+2}(Tij)+w_{m}(T)$, for $T = S \cap [n]$ and $m = |T|$.
                  Here is where we need to ensure supermodularity, i.e., $w_{m+2}(Tij)+w_{m}(T) \geq w_{m+1}(Ti) + w_{m+1}(Tj)$. 
      Since we have the freedom to choose each $w_d$ up to addition of a constant, we can achieve this by adding to each $w_d$ a suitably scaled convex function on $d$. 
		  As such a transformation preserves the compression $u$, we can ensure that $\mu(Sik) +\mu(Sjl)$ does attain the minimum.
                  
		\item $|\{i,j,k,l\}\cap[n]| = 3$. The positive Pl\"ucker relation here is equivalent to the positive 3-term incidence relation between $w_{|S\cap[n]|+1}$ and $w_{|S\cap[n]|+2}$.
		The proof of \Cref{thm:regular-2-skeleton} shows that the 3-term incidence relations are already implied by (HXM). 
		Further, (HXM+) strengthens this to imply \eqref{eq:positivepluecker}, since the terms of different sign agree under the correspondence of this implication, as seen in the discussion preceding the formulation of (HXM+).

		\item $|\{i,j,k,l\}\cap[n]| = 4$. In this case the Pl\"ucker relation is equivalent to a Pl\"ucker relation in $w_{|S\cap[n]|+2}$.
		By \Cref{lemma:1step_positive}, the positive variation of this relation follows from the last case. 
	\end{itemize}

        It was recently shown in~\cite{SpeyerWilliams:2021,ArkaniHamedLamSpradlin:2021} that the positive tropical Grassmannian equals the positive Dressian.
        Hence, $\mu$ is realizable by some totally positive subspace $L\in \Gr^+(n,2n)$. Consider $\pi$ to be the projection to the first $n$ coordinates and let $L_d= \pi(L\cap \{x_{n+d+1}=\dots = x_{2n} = 0\})$. The tropicalization of $L_d$ corresponds to the valuated matroid 
	\[
	\mu\backslash \{n+1,\dots, n+d\} / \{n+d+1,\dots, 2n\} = w_d \;.
	\]
	We have $L_1\subset \dots \subset L_n$, all of which are positive.
So $(w_1,\dots,w_d)$ is in the positive tropical flag variety.					
\end{proof}

Our second main result combines known facts with our Lemma~\ref{lem:positive-valuated-flag-matroid}.

\begin{theorem}\label{thm:positivity}
  Let $u \colon \Sym(n) \to \RR$.
  Then the following are equivalent:
	\begin{enumerate}
		\item The function $u$ is the compression of a valuated flag $(w_1,\dots, w_n)$ which can be realized by a totally positive flag of linear spaces.
		\item All polytopes in the subdivision induced by $w$ are Bruhat interval polytopes.
		\item The function $u$ satisfies (HXE) and (SQR) from \Cref{thm:regular-2-skeleton} as well as (HXM+).
	\end{enumerate}
\end{theorem}

\begin{proof}
	\enquote{$(1)\rightarrow (2)$}.
	Consider a totally positive flag $\cL=(L_1,\dots, L_n)$ that realizes $(w_1,\dots,w_n)$, e.g., over the field of Puiseux series, or a suitable extension for irrational coefficients; see \cite[\S2.6]{ETC}.
	For any polytope $P$ in the subdivision induced by $w$, we can obtain a flag $\tilde{\cL} =(\tilde{L}_1,\dots,\tilde{L}_n)$ that realizes the flag matroid corresponding to $P$ by taking a suitable rescaling of the flag $\cL$ and taking its quotient to the residue field, which is a subfield of $\RR$. 
	By construction, $\tilde{\cL}$ is in the totally non-negative part; yet some Pl\"ucker coordinates may vanish in the quotient. 
	Due to \Cref{prop:bruhat}, $P$ is a Bruhat polytope.
	
	\enquote{$(2) \rightarrow (3)$}.
	By \Cref{thm:regular-2-skeleton} and \Cref{prop:bruhat}, if the subdivision consists of Bruhat polytopes then (HXE) and (SQR) are satisfied.
	Suppose (HXM+) fails, so there is a hexagon where the maximum is not attained by the diagonal connecting the lowest and largest terms in the Bruhat order. 
	Then this diagonal appears in the subdivision and it is not a Bruhat interval polytope; see Figure~\ref{fig:bruhat}. 
	
	\enquote{$(3) \rightarrow (1)$}.
	By \Cref{thm:regular-2-skeleton} the function $u$ is the compression of the valuated flag matroid $(w_1,\dots, w_n)$.
        Hence, we can apply \Cref{lem:positive-valuated-flag-matroid}, which concludes this direction. 
\end{proof}

\begin{example}\label{ex.3}
We conclude \Cref{ex.2}. 
The Pl\"ucker coordinates computed in \Cref{ex.1} are nonnegative.
That is, the coefficient of $t^{\val(q)}$ is positive.
Therefore  $L_{1}\subset L_{2}\subset L_{3}$ corresponds to a point in the totally nonnegative flag variety \cite{Boretsky:2021}.
One can also check that $w$ satisfies (HXM+).
That is,
\[ \max\left\{w((1,2,3))+w((3,2,1)),w((2,1,3))+w((2,3,1)),w((1,3,2))+w((3,1,2))\right\} \]
attains its maximum at the $w((1,2,3))+w((3,2,1))$ entry, corresponding to the maximal and minimal elements of $\Sym(3)$ in the Bruhat order.
Then, by \Cref{prop:bruhat} and \Cref{thm:positivity}, $w$ subdivides $\Pi_{3}$ into Bruhat interval polytopes (see \Cref{ex:bruhat}).
\end{example}
	
\section{Concluding Remarks}

In \Cref{thm:regular-2-skeleton} we showed that the permutahedral subdivisions of a permutahedron can be recognized by looking at the $2$-skeleton.
Yet there are distinct permutahedral subdivisions of $\Pi_4$ with the same induced subdivision on the $2$-skeleton; cf.\ \Cref{fig:2-skeleton}.
We call the latter a \emph{permutahedral pattern} on the $2$-skeleton.
\begin{question}
  For a given permutahedral pattern on the $2$-skeleton of $\Pi_n$, what is the set of all permutahedral subdivisions of $\Pi_n$ inducing that pattern?
\end{question}

\smallskip

In Section~\ref{sec:computations} we computed the fans of permutahedral subdivisions of $\Pi_3$ and $\Pi_4$.
\begin{question}
  What can be said in general about the topological type of the spherical polytopal complex formed by the height functions of permutahedral subdivisions of $\Pi_n$?
\end{question}
Note that this complex is essentially the compression of the flag Dressian $\FlDr(n)$; see also \cite{BLMM:2017}.
There is reason to believe that the question above may be difficult.
First, the homology of the ordinary Dressians is only known for a few small cases; see \cite[\S5.4]{Tropical+Book}
Second, already the secondary fan of $\Pi_4$ is out of reach of current software tools such as \TOPCOM \cite{TOPCOM-paper}, \mptopcom \cite{JordanJoswigKastner:2018} or \Gfan \cite{gfan}.

\smallskip

While here we analyzed the regular permutahedra and their subpermutahedra, it is most natural to extend this line of work to generalized permutahedra.
\begin{question}
  How does the characterization in \Cref{thm:regular-2-skeleton} generalize to permutahedral subdivisions of arbitrary generalized permutahedra?
\end{question}

\acknowledgements
We are indebted to Lars Kastner for help with the computations.
Further, we want to thank Jonathan Boretsky, Federico Castillo, Daniel Corey, Christopher Eur, Melissa Sherman-Bennett and Ben Smith for their feedback and Lauren Williams for pointing us to~\cite{Boretsky:2021} and for sharing insights regarding Bruhat interval polytopes.
MJ and GL are grateful to the Hausdorff Research Institute for Mathematics, Bonn for their hospitality during the Trimester Program \enquote{Discrete Optimization}.

This paper is the full version of two extended abstracts (with the same title) for the conferences \enquote{34th International Conference on Formal Power Series \& Algebraic Combinatorics}, Bangalore, India, and \enquote{Discrete Mathematics Days 2022}, Santander, Spain.

\printbibliography
\end{document}